\documentclass[oneside,11pt]{article}
\usepackage{epsfig, amsfonts, amsmath, amsthm,amsbsy}
\usepackage{color}
\usepackage{subfigure}
\newtheorem{theorem}{Theorem}

\newtheorem{remark}{Remark}
\newtheorem{example}{Example}

\def\ve{\varepsilon}
\def\vr{\varepsilon}
\def\ds{\displaystyle}
 \DeclareMathOperator\erfc{erfc}

\begin{document}

 \title{\bf Parameter-uniform approximations for a singularly perturbed convection-diffusion problem with a  discontinuous initial condition\thanks{This research was partially supported  by the Institute of Mathematics and Applications (IUMA), the projects PID2019-105979GB-I00 and PGC2018-094341-B-I00 and the Diputaci\'on General de Arag\'on (E24-17R).}}
\author{J.L. Gracia\thanks{Department of Applied Mathematics, University of
Zaragoza, Spain.\  email: jlgracia@unizar.es} \and E.\ O'Riordan\thanks{School of Mathematical Sciences, Dublin City
University, Dublin 9, Ireland.\ email: eugene.oriordan@dcu.ie}}

\maketitle

\begin{abstract}
A singularly perturbed parabolic problem of convection-diffusion type with a discontinuous initial condition is examined.
A particular complimentary error function is identified which matches the discontinuity in the initial condition. The difference between this analytical function and the solution of the parabolic problem is approximated numerically.
A co-ordinate  transformation is used so that a layer-adapted mesh can be aligned to the interior layer present in the solution.    Numerical analysis is presented for the associated numerical method, which establishes that the numerical method is a parameter-uniform numerical method. Numerical results are presented to illustrate the pointwise error bounds  established in the paper.
\end{abstract}



%
\section{Introduction}

In this paper, we examine a singularly perturbed convection-diffusion problem with a discontinuous initial condition of the form: Find $\hat u$ such that
\begin{subequations}\label{problem class}
\begin{eqnarray}
-\ve \hat u_{ss} + \hat a \hat u_s +\hat b \hat u+\hat u_t = \hat f, \quad (s,t) \in  \hat Q:= (0,1) \times (0,T]; \label{pde}\\
  \hat u(s,0) = \phi (s) \not \in C^0(0,1); \quad \hat a >0; \ \hat b\geq 0,
\end{eqnarray}
\end{subequations}
 with Dirichlet boundary conditions. As this is a parabolic problem, an interior layer emerges from the initial  discontinuity, which is  diffused over time if $\ve =O(1)$.  However, when the parameter is small, the interior layer is convected along a characteristic curve associated with the reduced problem.

In  \cite{apnum}, we examined a related singularly perturbed reaction-diffusion problem (set  $\hat a\equiv 0$ in (\ref{problem class})) with a discontinuous initial condition  and we used an idea from \cite{Flyer} to first identify  an analytical function which matched the discontinuity in the initial condition and also satisfied a constant coefficient version of the differential equation. A numerical method was then constructed to approximate the difference between the solution of the singularly perturbed reaction-diffusion problem and this analytical function. The numerical approximation involves approximating an  interior layer function whose location, in the case of a reaction-diffusion problem, is fixed in time. In the corresponding convection-diffusion problem, the location of the interior layer function moves in time and, from \cite{disc-cd1}, we know that  the numerical method needs to track this location.  Shishkin \cite{shishkin4} examined problem (\ref{problem class}) in the case where  the initial condition $\phi \in C^0(0,1) \setminus C^1(0,1)$. In \cite[Chapter 10 and \S 14.2]{Shish-redbook}, Shishkin and Shishkina discuss the {\it method of additive splitting of singularities} for singularly perturbed problems with non-smooth data. We follow the same philosophy here.

When the convective coefficient  depends solely on time ($\hat a(s,t) \equiv \hat a(t) >0$), the main singularity generated by the discontinuous initial condition can be explicitly identified by a particular complimentary  error function. This error function tracks the location of the interior layer emanating from the discontinuity in the initial condition  and it also satisfies the homogenous partial differential equation ({\ref{pde}) exactly. When this discontinuous error function is subtracted from the solution $\hat u$ of (\ref{problem class}), the remaining function (denoted below by $\hat y$)  contains no interior layer and it can be adequately approximated numerically  by designing a numerical method which incorporates a Shishkin mesh in the vicinity of  the boundary layer \cite{arxiv}.

In this paper we deal with the more general case of the convective coefficient depending on both space and time.
In this case, the situation is more complicated. The main singularity  is again  a particular complimentary error function which tracks the location of the interior layer, but when the coefficient $\hat a$ in ({\ref{pde}) varies in space
this complimentary error function  does not satisfy the homogenous partial differential equation  ({\ref{pde}). Moreover, when this discontinuous error function is subtracted from the solution $\hat u$ of (\ref{problem class}), the remaining function $\hat y(s,t)$ contains its own  interior layer. To generate an accurate numerical approximation to this  remainder $\hat y$, a  coordinate transformation is first required in order  that a mesh can be constructed to track the location of this internal layer.  Hence the numerical method used to approximate the  remainder  (when $\hat a$ depends on space  and time) is different to the numerical method used to approximate the  remainder  in the case of the convective coefficient solely depending on time. Needless to say, the more general method can also be applied to the case where the convective coefficient $\hat a$ is independent of space.
If the coordinate transformation is not used, in the numerical section we demonstrate  that one does not generate a parameter-uniform  approximation if  $\hat a$ depends on the space variable.

In  \S \ref{sec:continuous} we specify the continuous problem and deduce bounds on the partial derivatives of the solution. Some of the more technical details involved in the proofs of the bounds on the continuous solution are presented in the appendices. A piecewise-uniform mesh is constructed in \S \ref{sec:Method+Error}, which is designed to be refined in the neighbourhood of the curve $\Gamma ^*$, which identifies the location of the interior layer at each time. To analyse the parameter-uniform convergence of the resulting numerical approximations on such a mesh, it is more convenient to perform the analysis in a transformed domain where the location of the interior layer is fixed in time. To simplify the discussion of the method and the associated numerical analysis,  we discuss the case where there is no source term present in the problem in \S \ref{sec:continuous} and  \S \ref{sec:Method+Error}.  In \S \ref{source-term}, we outline the modifications required when  a source term is present.  In \S \ref{sec:numerical}, we present some numerical results to illustrate the performance of the method.

{\bf Notation:} Throughout the paper, $C$ denotes a generic constant that is independent of the singular perturbation parameter $\vr$ and all the discretization parameters. The $L_\infty$ norm on the domain $D$ will be denoted by $\Vert \cdot \Vert_D$  and the subscript is omitted if $D=\hat Q.$ We also define the jump of a function at a point $d$ by$[\phi ] (d) := \phi (d^+) - \phi (d^-) $. Functions defined in the computational domain will be denoted by $f(x,t)$ and functions defined in the untransformed domain will be denoted by $\hat f(s,t)$.

\section{Continuous problem} \label{sec:continuous}

Consider the following convection-diffusion problem\footnote{As in \cite{friedman}, we define the space ${\mathcal C}^{0+\gamma}(D )$, where $D \subset \mathbf{R}^2$ is an open set, as the set of all functions that are H\"{o}lder continuous of degree $\gamma \in (0,1) $ with respect to the metric $\Vert \cdot \Vert, $
where for all ${\bf p}_i=(x_i,t_i),  \in \mathbf{R}^2, i=1,2; \
\Vert {\bf p}_1- {\bf p}_2 \Vert^2  = (x_1-x_2)^2 + \vert t_1 -t_2 \vert$.
For $f$ to be in ${\mathcal C}^{0+\gamma}(D ) $ the following semi-norm needs to be finite
\[
\lceil f \rceil _{0+\gamma , D} := \sup _{{\bf p}_1
 \neq {\bf p}_2, \ {\bf p}_1, {\bf p}_2\in D}
\frac{\vert f({\bf p}_1) - f({\bf p}_2) \vert}{\Vert {\bf p}_1- {\bf p}_2 \Vert ^\gamma} .
\] The space ${\mathcal C}^{n+ \gamma}(D ) $ is defined by
\[
{\mathcal C}^{n+\gamma }( D ) := \left \{ z : \frac{\partial ^{i+j} z}{
\partial x^i
\partial t^j } \in {\mathcal C}^{0+\gamma }(D), \ 0 \leq i+2j \leq n \right \},
\]
and $\Vert \cdot \Vert _{n + \gamma}, \ \lceil \cdot \rceil _{n+\gamma}$ are the associated norms and semi-norms.
}: Find $\hat u$ such that
\begin{subequations}\label{problem3a}
\begin{align}
& \hat L \hat u:= -\ve \hat{u}_{ss} + \hat a(s,t) \hat{u}_s +\hat{u}_t= \hat f, \quad (s,t) \in \hat Q:=(0,1)\times (0,T], \label{a(x,t)}\\
&\hat{u}(s,0)=\phi (s), \, 0 \leq s \leq 1;  \   [\phi ](d)\neq 0,\ 0 < d =O(1) <1; \label{1b} \\
&\hat{u}(p,t)=0,  \ 0 < t \leq T, \ p=0,1;\label{1c}\\
&\hat a(s,t) > \alpha >0, \ \forall (s,t) \in \hat Q, \quad \hat a, \, \hat f  \in C^{4+\gamma } (\bar{\hat Q} ); \label{1d}\\
&\phi ^{(2i)} (p) =0;\  i=0,1,2; \ p=0,1; \ \phi \in C^4((0,1); \setminus \{d \});\label{comp0}\\
& \hat f^{(i+2j) }(p,0)=0; \ 0 \leq i+2j \leq 4-2p, \  p=0,1; \label{comp1}\\
& \hat a_s(d,0) =0, \quad [\phi ' ](d)=  0. \label{simplify}
\end{align}
 In general,  a moving interior layer and a boundary layer  will  appear in the solution.   When the convective term depends on space then the path of the characteristic curve $\hat \Gamma ^*$ (associated with the reduced problem) is implicitly defined by
\begin{equation} \label{characteristic}
 \hat \Gamma^*:= \{ (d(t),t) \vert d'(t)=\hat a(d(t),t), \quad d(0)=d \}.
\end{equation}
Since we have assumed that $\hat a>0$, the function $d(t)$ is monotonically increasing. We restrict the size of the final time $T$ so that the interior layer does not interact with the boundary layer. Thus,  we limit the final time $T$ \footnote{In \cite{arxiv} we examine the effect of not restricting the final time $T$.} such that
\begin{equation}\label{delta}
 1 > \delta := \frac{1-d(T)}{ 1-d} >0.
\end{equation}
 In the error analysis, we are required to impose a further restriction on the final time by assuming that
\begin{equation}\label{T-limit}
 \frac{2T}{\delta} \Vert \hat a_s  \Vert  \leq 1-\gamma , \ 0 < \gamma < 1.
\end{equation}
\end{subequations}

The discontinuity in the initial condition generates an interior layer emanating from the point $(d.0)$. By identifying the leading term $ 0.5 [\phi ](d)\hat \psi _0 $ in an asymptotic expansion of the solution, we  can  define the continuous function
\begin{equation}\label{def-y}
\hat y (s,t) := \hat u(s,t)   - 0.5 [\phi ](d)\hat \psi _0 (s,t), \  \hat \psi _0(s,t) := \erfc \left ( \frac{d(t)-s}{2\sqrt{\ve t}} \right),
\end{equation}
where
\begin{eqnarray}\label{untransformed}
\hat L\hat y&=  \hat f + 0.5 [\phi ](d) \bigl(\hat a(d(t),t)- \hat a(s,t)\bigr) \frac{\partial}{\partial s}\hat \psi_0(s,t).
\end{eqnarray}

Note that in (\ref{simplify}) we impose the constraint  $[\phi ' ](d)=  0$ on the initial condition. This assumption permits us to complete the  analysis of the numerical error.
Based on the expansion (\ref{expansion2}) of the solution derived in the appendix, we note that $\hat \psi _i\in C^{i-1}(\bar {\hat  Q}) , i\ge 1,$ which implies (due to assumption  (\ref{simplify})) that $\hat y (s,t) \in C^1(\bar {\hat  Q})$. Moreover, if the constraint $[\phi ' ](d)=  0$ is not imposed, then there is a reduction in the order of convergence of the numerical approximations as in  \cite[Theorem 1]{arxiv},  \cite{shishkin4};  and the error analysis remains an open question when $[\phi ' ](d) \neq  0$.

 In addition, in (\ref{simplify}) we also assume that $\hat a_s(d,0) =0$, which results in the interior layer function  (defined in~\eqref{decomposition_y}) being sufficiently regular to establish the bounds  (\ref{bds-final}).
The constraint (\ref{T-limit})  is used in establishing the pointwise bound (\ref{bds-z-a}) on the interior layer function. This bound  is used to determine  the transition points in the Shishkin mesh around the interior layer.
Finally, for sufficiently smooth and  compatible boundary conditions at $(0,0)$ and $(1,0)$, there is no loss in generality in assuming the constraints  (\ref{1c}), as the simple subtraction of the linear function $\hat q(s,t):=\hat u(0,t) (1-s) +\hat u(1,t) s$ from $\hat u$ leads us to problem (\ref{problem3a}) with $\hat f$ replaced by  $\hat f_1:=\hat f- \hat L \hat q$.

Observe that the inhomogeneous term in (\ref{untransformed}) is continuous, but not in $ C^1(\bar {\hat  Q})$ on the closed domain.
The presence of this inhomogeneous term  will induce an interior layer into the function $\hat y$.
So if the convective coefficient $\hat a(s,t)$ depends on the space variable, we are required to transform the problem (\ref{problem3a}) so that  the curve $ \hat \Gamma ^*$ is transformed to a straight line, around which a piecewise-uniform Shishkin mesh is constructed.

One possible choice \cite{disc-cd1} for the transformation  $X:(s,t) \rightarrow (x,t)$ is the piecewise linear map given by
\begin{equation}\label{map}
x(s,t):= \left\{ \begin{array}{ll}
\displaystyle
\frac{d}{d(t)}s, &\quad s \leq d(t), \\[2ex]
\displaystyle
1- \frac{1-d}{1-d(t)}(1-s), & \quad s \geq d(t), \end{array}
\right.
\end{equation}
which means that $\hat a(d(t),t) =a(d,t)$.  Define the left and right subdomains to be
\[
Q^-:=(0,d) \times (0,T] \quad \hbox{and} \quad Q^+:=(d,1) \times (0,T].
\]
Using this map the problem to solve numerically, transforms into the problem: Find $y$ such that
\begin{subequations}\label{problem3}
\begin{align}\label{de-problem3}
& {\cal{L}}y
= g \left( f + 0.5 [\phi ](d)  \frac{(a(d,t)- a(x,t))}{\sqrt{\ve \pi t}} e^{-\frac{g(x,t)(x-d)^2}{4\ve t}}\right), \quad x \neq d,\\
& [y \ ](d,t)  =0, \quad \ \left [\frac{1}{\sqrt{g}}y _x\right](d,t) =0,\\
& y(p,t)  =  - 0.5 [\phi ](d)\hat \psi _0 (p,t), \  p=0,1, \ 0 < t \leq T, \\
& y(x,0)  = \begin{cases}
\displaystyle
\phi (x) , & x < d, \\[1ex]
\displaystyle
\phi (d^-) , & x = d, \\[1ex]
\displaystyle
\phi (x) - [\phi ](d), & x > d. \end{cases}
\\
& \hbox{where } {\cal{L}}y  :=  -\ve y_{xx} +\kappa (x,t)y_x+ g(x,t) y_t,\ \hbox{and the coefficients are} \nonumber
\end{align}
\begin{align}\label{coeff-problem3}
\kappa (x,t) & := \sqrt{g}\bigl(a(x,t)+a(d,t)(\psi _d(x)-1)\bigr), \
  \\ \nonumber
 \\
 \psi _d(x)& :=  \begin{cases}
\displaystyle \frac{ d-x}{d}, & x < d, \\[2ex]
\displaystyle \frac{x-d}{1-d}, & x > d.
\end{cases}
 \quad g(x,t) :=  \begin{cases}
\displaystyle
 \left(\frac{d(t)}{d}\right)^2, & x < d, \\[2ex]
\displaystyle
\left(\frac{1-d(t)}{1-d} \right)^2, & x > d.
\end{cases}
 \label{def-g}
\end{align}
\end{subequations}
Observe that  $g$ is a discontinuous function along $x=d$ and  $[g](d,t) <0 $ for all $t >0 $. In addition,  for all $t \geq 0$,  $\vert g -1 \vert \leq C\vert d(t) -d \vert \leq Ct$ and
\begin{subequations}
\begin{eqnarray}
1 \leq \sqrt{g} \leq 1 +\frac{T\Vert a \Vert}{d}, \ x \leq d \quad , \quad \delta \leq \sqrt{g} \leq 1,\ x \geq d,
 \label{bounds-g}
\end{eqnarray}
\end{subequations}
The transmission condition $ \ [\frac{1}{\sqrt{g}}y _x](d,t) =0$ corresponds to $ \ [\hat y _s](d(t),t) =0$.
Note that there exists a positive constant $A$, such that
\begin{equation}\label{def-A}
\vert \kappa (x,t) \vert\leq  A \vert d-x \vert,\ A:= \left(1+\frac{T\Vert a \Vert}{d}\right)\left( \Vert a_x \Vert +   \Vert a \Vert\max \left\{ \frac{1}{d},\frac{1}{1-d} \right\} \right).
\end{equation}
We  associate the following differential operator
\[
{\cal L}'_\ve \omega (x,t) :=
\begin{cases}
\omega (x,t), &  x=0,1, \, t\ge 0, \\
 \omega (x,0), &  x\in(0,d)\cup(d,1),\\
-\ve \omega_{xx} + \kappa (x,t)\omega_x +g(x,t) \omega_t, & x\neq d ,\ t > 0,\\
\ -\left[\frac{1}{\sqrt{g}} \omega _x \right]& x=d,\ t\geq 0,
\end{cases}
\]
with this transformed problem. For this operator ${\cal L}'_\ve$ a comparison principle holds  \cite{disc-cd1}.
\begin{theorem}\cite{disc-cd1} \label{max_princ}
Assume that a function $\omega \in  {\cal C}^0(\bar Q) \cap {\cal C}^{2} (Q ^- \cup Q ^+ )$  satisfies
${\cal L}'_\ve \omega(x,t)  \geq  0, \ \mbox{for all} \  (x,t)\in  \bar Q$  then $ \omega(x,t)  \geq  0$, for all $ (x,t) \in  \bar Q$.
\end{theorem}
Using this comparison principle we see from (\ref{def-A}) that
\[
\vert y(x,t) \vert \leq  \frac{A}{\delta^2 } (1+ \Vert f \Vert) t + \Vert \phi \Vert_{\bar Q^- \cup \bar Q^+} + \vert[\phi](d) \vert +  \left\{ \begin{array}{ll}
\displaystyle \frac{x}{d}, & x \leq d, \\[2ex]
\displaystyle \frac{1-x}{1-d}, & x \geq d. \end{array}
\right.
\]
That is, $\Vert y \Vert \leq C$.

The solution of problem  (\ref{problem3}) can be decomposed into the sum of a regular component $v$, a boundary  layer component $w$, a weakly singular component and an interior layer $z$ component:
\begin{equation}\label{decomposition_y}
y=v+w+ 0.5 \sum _{i=2}^4 [\phi ^{(i)} ] (d)  \frac{(-1)^i}{i!}   \psi _i + z.
\end{equation}
In Appendix B, the regular component $\hat v \in C^{4+\gamma}(\hat Q)$ and the boundary layer component $\hat w \in C^{4+\gamma}(\hat Q)$
are defined in the original variables $(s,t)$. The mapping $X:(s,t) \rightarrow (x,t)$  defined in (\ref{map}) is not smooth along the interface $x=d$. Hence, in the transformed variables the regular component $v$ is defined  so that $v \in (C^{4+\gamma}(\bar Q^+) \cup C^{4+\gamma}(\bar Q^-)) \cap C^1(\bar{\hat Q})$ and   satisfies the  bounds
\[
\Bigl \vert \frac{\partial ^{i+j}}{\partial x ^i \partial t ^j } v(x,t)\Bigr \vert  \leq  C , \ 0 \leq i+j \leq  2;\quad
\Bigl \vert \frac{\partial ^{3}    }{\partial x ^3  } v(x,t)\Bigr \vert  \leq  C(1+\ve ^{-1}) ;\qquad  x \neq d.
\]
Also, the boundary layer function $ w \in (C^{4+\gamma}(\bar Q^+) \cup C^{4+\gamma}(\bar Q^-)) \cap C^1(\bar {\hat Q})$ and  satisfies the bounds~\cite[bound in (9)]{disc-cd1}
\begin{equation}\label{lay_comp}
\Bigl \vert \frac{\partial ^{j+m}w}{\partial x^j \partial t^m} (x,t) \Bigr \vert \leq C  \ve^{-j} (1+\ve ^{1-m}) e^{-\frac{ \alpha \delta(1-x)}{2 \ve}}, \ 0\leq j\leq 3, \ m=1,2.
\end{equation}
 As $y,v,w$  and $\psi _i, \, i=2,3,4$ are all bounded, then the  interior layer function $z$ is also bounded.

\begin{remark} \label{rem:z}
We note that if $\hat a(s,t) =a(t)$, then $z  \equiv 0$ and the coordinate transformation is not needed for this problem class. Error estimates and extensive numerical results for this problem class are given in~\cite{arxiv}.
\end{remark}

\begin{theorem} \label{th_z_component}
The interior layer component  $ z \in C^{2+\gamma}(\bar Q^-)\cup C^{2+\gamma}(\bar Q^+)$
 satisfies the bounds
\begin{equation}
 \vert z(x,t) \vert \leq Ce^{-\frac{\gamma g(t) (d-x)^2}{4\ve t}} ,\quad  (x,t) \in  Q. \label{bds-z-a}
\end{equation}
In addition, for $x \neq d$,
\begin{subequations} \label{bds-final}
 \begin{align}
 \Bigl \Vert \frac{\partial ^{i+j} z}{\partial x ^i \partial t ^j }  \Bigr \Vert  &\leq C\bigl( 1+\ve^{-i/2} \bigr), \ i+2j\leq 3;\label{bds-z-c}\\
 \Bigl \vert \frac{\partial ^2  z}{\partial t^2 }  (x,t) \Bigr \vert &\leq C \left( 1+\sqrt{\frac{\ve }{ t}}\right). \label{bds-z-d}
\end{align}
\end{subequations}
\end{theorem}

\begin{proof}
The  interior layer function $ z$   is decomposed into the sum of two subcomponents
\begin{equation} \label{decomposition_z}
 z=z_c+ 0.5 [\phi ](d)   z_p,
\end{equation}
where $z_c$ satisfies the problem
\begin{subequations}\label{zc-def}
\begin{equation}
{\cal{L}}z_c= -0.5 \sum _{i=2}^4 [\phi ^{(i)} ] (d)  \frac{(-1)^i}{i!}  {\cal{L}} \psi _i (x,t), \ x \neq d,
\end{equation}
\begin{align}
z_c(x,0)=  z_c(0,t) =z_c(1,t) =0;\ [z_c](d,t)=
\left[\frac{1}{\sqrt{g}}\frac{\partial z_c}{\partial x}\right](d,t)=0,
\end{align}
\end{subequations}
and  $z_p$ satisfies the problem
\begin{subequations}\label{z-def}
\begin{align}
{\cal{L}}z_p=  \bigl(a(d,t)- a(x,t)\bigr) \frac{g(x,t)}{\sqrt{\ve \pi t}}  e^{-\frac{g(x,t)(x-d)^2}{4\ve t}}, \ x \neq d;
\\
z_p(x,0)=
z_p(0,t) =z_p(1,t)= 0;\  [z_p](d,t)=0 ; \left[\frac{1}{\sqrt{g}}\frac{\partial  z_p }{\partial x}\right](d,t)=0. \label{jump-z}
\end{align}
\end{subequations}
In  Appendix C, the subcomponent  $z_p$  is further decomposed into the sum~\eqref{DecompZP}
\[
z_p = z_q+ z_R,
 \]
where it is established that $z_R \in C^{4+\gamma}(\bar Q^-)\cup C^{4+\gamma}(\bar Q^+)$ and the weakly singular function $ z_q \in C^{2+\gamma}(\bar Q^-)\cup C^{2+\gamma}(\bar Q^+)$  is explicitly identified  in~\eqref{ZQ}.
Bounds on the derivatives of  the subcomponent $z_q$  are also  given in~\eqref{ZQxt}. Moreover, it is established  in  (\ref{bds-zc}) and  (\ref{bds-zr}) that
\[
\vert {\cal{L}}z_c (x,t) \vert \leq C\sqrt{\ve}  E_\gamma (x,t) \quad \hbox{and } \quad \vert {\cal{L}}z_R (x,t) \vert \leq C   E_\gamma (x,t)
\]
where $E_\gamma (x,t) := e^{-\frac{\gamma g(t) (d-x)^2}{4\ve t}}.$
 From (\ref{T-limit})  and~\eqref{bounds-g}, note the following
\begin{align*}
  {\cal{L}}E_\gamma&= \frac{\gamma gE_\gamma }{2t}\left(1+
  (1-\gamma)\frac{g(d-x)^2}{2\ve  t}+(a(x,t)+a(d,t)(\psi_d(x)-1))\frac{\sqrt{g} (d-x)}{\ve }
 \right) \\
  & \ge \frac{\gamma gE_\gamma }{2t}\left(1+
  (1-\gamma)\frac{g(d-x)^2}{2\ve  t}+(a(x,t)-a(d,t))\frac{\sqrt{g} (d-x)}{\ve }
 \right)\\
&\geq \frac{\gamma g E_\gamma }{2t}\left(1+
 \left[ (1-\gamma)- \frac{2T}{\sqrt{g}} \Vert \hat a_s  \Vert)\right] \frac{g(d-x)^2}{2\ve t}
 \right) \\
&\geq \frac{\gamma g E_\gamma }{2t} \geq \frac{\gamma  \delta^2  E_\gamma }{2T}.
\end{align*}
Using a comparison principle seperately on each subdomain $Q^-$ and $Q^+$, we can then obtain the bounds
\[
\vert z_c (x,t) \vert \leq  C E_\gamma (x,t), \quad \vert z_R (x,t) \vert \leq  C E_\gamma (x,t). 
\]
Combining this bound with the bounds on  $z_q(x,t)$ (from~\eqref{ZQxt} in the final  Appendix C) we achieve the pointwise bound in (\ref{bds-z-a}).

We  transform the problems   ${\cal{L}}z_c (x,t) =: F_c(x,t)$, ${\cal{L}}z_R (x,t) =: F_R(x,t)$ back to the original variables
\[
\hat L\hat z_c (s,t) = \hat F_c(s,t),\quad \hbox{and} \quad \hat L\hat z_R (s,t) = \hat F_R(s,t),
\]
and now apply the standard  argument from \cite[pg.352]{ladyz} , separately on the subdomains $ Q^-$ and  $Q^+$,  to deduce the remaining bounds.
\end{proof}

\section{ Numerical method in the transformed domain and associated error analysis} \label{sec:Method+Error}

We approximate the solution of problem (\ref{problem3}) on a rectangular grid in the computational domain
 $ \bar Q^{N,M}=\{x_i\}^N_{i=0} \times \{t_j\}_{j=0}^M $ which concentrates mesh points in the interior and boundary layers.  We denote by $\partial Q^{N,M}:=\bar Q^{N,M}\backslash Q.$ The mesh $\bar Q^{N,M}$ incorporates a uniform mesh  ($t_j:=k j$ with $k=T/M$) for the time variable and the  grid points for the space variable are distributed by means of a piecewise uniform Shishkin mesh  with $h_i:= x_i-x_{i-1}$.
 Based on the bounds (\ref{lay_comp}) and  (\ref{bds-z-a}) on the layer components, this mesh is defined with respect to the transition points
\begin{subequations}\label{transition points}
\begin{align}
\sigma _1 &:=\min \left\{\ds\frac{d}{4},  2\sqrt{T\ve} \ln N \right\}, \
 \sigma _2 := \min  \left \{ 1-d(T), \frac{d}{4}, 2\sqrt{\frac{T\ve}{\delta}} \ln N \right \},\\
\sigma &:= \min  \left\{ \frac{1-(d+\sigma_2)}{2}, \frac{2 \ve}{\alpha \delta} \ln N \right\},
\end{align}
\end{subequations}
which split  the interval $[0,1]$ into the five subdomains
\begin{equation}
[0,d-\sigma_1]\cup[d-\sigma_1,d] \cup[d, d+\sigma _2] \cup [d +\sigma_2, 1 -\sigma ] \cup [1- \sigma,1].
\end{equation}
The grid points are uniformly distributed within each subinterval in the ratio
$\frac{3N}{8}:\frac{N}{8}:\frac{N}{8}:\frac{N}{4}:\frac{N}{8}$.
We discretize problem (\ref{problem3}) using an Euler method to approximate the
time variable and an upwind finite difference operator to approximate in space. Hence the discrete problem\footnote{We use the following notation for various finite difference operators:
\begin{align*}
 aD_x Y (x_i,t_j) := 0.5 (a(x_i,t_j)+\vert a (x_i,t_j)\vert)D^-_x Y(x_i,t_j) + 0.5 (a(x_i,t_j)-\vert a (x_i,t_j)\vert)D^+_x Y (x_i,t_j)  \\
D^-_t Y (x_i,t_j) := \ds\frac{Y(x_i,t_j)-Y(x_i,t_{j-1})}{k}, \quad
D^-_x Y(x_i,t_j) :=\ds\frac{Y(x_i,t_j)-Y (x_{i-1},t_j)}{h_i}, \\
D^+_x Y (x_i,t_j)  :=\ds\frac{Y(x_{i+1},t_j)-Y(x_i,t_j)}{h_{i+1}}, \
 \delta^2_x Y(x_i,t_j) := \ds\frac{2}{h_i+h_{i+1}}(D_x^+Y(x_i,t_j)-D^-_x Y(x_i,t_j)).
 \end{align*}}}
 is: Find $Y$ such that
\begin{subequations} \label{discrete-problem}
\begin{align}\label{discrete-problem_interior}
(-\ve  \delta ^2_x+\kappa D_x +   g D^-_t )Y &= {\cal{L}}y(x_i,t_j), \, x_i \neq   d,  \, t_j >0, \quad
\\
\ \left[\frac{1}{\sqrt{g}}D_xY\right] (d,t_j) &= 0, \quad  \, x_i=d, t_j >0,\\
Y&=y(x_i,t_j),  \quad (x_i,t_j)\in \partial Q^{N,M};
\\
\hbox{where} \quad \ \left[\frac{1}{\sqrt{g}} D_xY \right] (d,t_j) &:= \frac{1-d}{1-d(t_j)}D^+_x Y(d,t_j) - \frac{d}{d(t_j)}D^-_x Y(d,t_j). \nonumber
\end{align}
\end{subequations}
Associated with this discrete problem is the upwinded finite difference operator: For any mesh function $U$,  define
\begin{eqnarray*}
L^{N,M}U (x_i,t_j):=\left\{
\begin{array}{lll}
(-\ve  \delta ^2_x +\kappa D_x +   g D^-_t ) U(x_i,t_j), & x_i \neq  d, \, t_j>0, \\
 -\ve \left[\frac{1}{\sqrt{g}}D_xU\right] (x_i,t_j), & x_i =  d, \, t_j>0,\\
U(x_i,t_j), &  (x_i,t_j) \in \partial Q^{N,M}.
\end{array}
\right.
 \end{eqnarray*}
 This discrete operator satisfies a discrete comparison principle \cite{disc-cd1} and we can then establish that
\[
\vert Y(x_i,t_j) \vert \leq \frac{A}{\delta^2 } (1+ \Vert f \Vert) t_j +\Vert \phi \Vert_{\bar Q^- \cup \bar Q^+}  + \vert[\phi](d) \vert +   \left\{ \begin{array}{ll}
\displaystyle \frac{x_i}{d}, & x_i \leq d, \\[2ex]
\displaystyle \frac{1-x_i}{1-d}, & x_i \geq d. \end{array}
\right.
\]
Hence $\Vert Y \Vert_{\bar{Q}^{N,M}}\le C.$
To perform the error analysis the  discrete solution is decomposed into the sum
\[
 Y=V+W +0.5 \sum _{i=2}^4 [\phi ^{(i)} ] (d)  \frac{(-1)^i}{i!}  \Psi _i + Z;
\]
where $V$ and $W$ are the discrete counterparts to $v$ and $w$. Using a standard argument \cite{fhmos} one can establish that
\begin{equation}\label{smooth-bnd}
\Vert v+w- (V+W) \Vert _{\bar Q^{N,M}} \leq  C N^{-1} \ln N  + CM^{-1}.
\end{equation}
For the remainder of the numerical analysis we will assume that $\ve$ is sufficiently small so that
\[
\sigma _1 = \sigma _2 =2\sqrt{T\ve} \ln N, \quad \sigma = \frac{2 \ve}{\alpha \delta} \ln N .
\]
When this is not the case, the argument is classical as then $\ve ^{-1} \leq C \ln N $.

The additional terms   $\Psi _i , i=2,3,4;$ and  $Z$  are defined as follows:
For $i=2,3,4$
\begin{align*}
L^{N,M}\Psi _i&=  {\cal L}\psi _i \quad  x_i \neq  d, \, t_j>0;\\
\qquad \Psi _i &=\psi _i,  (x_i,t_j) \in \partial Q^{N,M}; \quad \left[\frac{1}{\sqrt{g}} D_x\Psi _i\right] (d,t_j)=0, \, t_j>0;
  \end{align*}
  and
  \begin{subequations} \label{disc-Z}
 \begin{align}
L^{N,M}Z & = {\cal{L}}z,\    x_i \neq  d, \, t_j>0;\\
  Z &=0,\  (x_i,t_j) \in \partial Q^{N,M}; \quad
 \left [\frac{1}{\sqrt{g}}D_xZ \right] (d,t_j)=0.
 \end{align}
 \end{subequations}
By the   discrete comparison principle, we have that  $\Vert \Psi _m \Vert \leq C(\sqrt{\ve } )^m, m=2,3,4$ and
we  can examine the truncation error for   $\psi _m,\  m=2,3,4$:
\begin{align*}
\vert L^{N,M} \bigl( \Psi _2 - \psi _2\bigr) (x_i,t_j) \vert &\leq  C\left(1+ \frac{\sqrt{\ve}}{\sqrt{  t_j}}\right) N^{-1}+C \left(1+\frac{\ve}{t_j}\right) M^{-1}, \quad x_i \neq d,
\\
\vert L^{N,M} \bigl( \Psi _3 - \psi _3\bigr) (x_i,t_j) \vert &\leq  C  N^{-1}+C\left(1+  \ve {\sqrt{\frac{\ve}{t_j}}}\right)M^{-1}, \  x_i \neq d,
\\
\vert L^{N,M} \bigl( \Psi _4 - \psi _4\bigr) (x_i,t_j) \vert &\leq  C  N^{-1}+CM^{-1}, \  x_i \neq d,
\\
\vert L^{N,M} \bigl( \Psi _m - \psi _m\bigr) (d,t_j) \vert &\leq C \sqrt{\ve} N^{-1} \ln N,\qquad m=2,3,4.
\end{align*}
 Applying the argument from  \cite{Zhemukhov1} (see \cite[Theorem 1]{arxiv} for more details)  we deduce that
 \begin{equation}\label{sing-bnd}
\vert (\Psi _m - \psi _m) \vert \leq C (N^{-1}\ln N + M^{-1}  \ln M), \quad m=2,3,4,
 \end{equation}
 where we have used the bounds established in Appendix A for the singular functions $\psi _m,\  m=2,3,4.$ From the proof of Theorem~\ref{th_z_component}, we have the bounds
\[
\vert L^{N,M}Z  (x_i,t_j) \vert , \vert z(x,t)\vert \leq C E_\gamma (x,t).
\]
 Also, as $\Vert Z \Vert \leq C$, we can use a discrete comparison separately on each subinterval to sharpen the bound on $Z(x_i,t_j)$.
\begin{theorem}\label{thm4}
For sufficiently large $N$ and $M \geq {\mathcal O}(\ln (N))$, the solution of (\ref{disc-Z}) satisfies the bounds
\begin{align*}
(a) \qquad \vert Z (x_i,t_j) \vert &\leq C
\frac{\prod _{n=1}^i \bigl(1+\frac{h_n}
{ \sqrt{2T\ve}}\bigr)  }{\prod _{n=1}^{N/2} \bigl(1+\frac{h_n}{ \sqrt{2T\ve}}\bigr)  }
+ C N^{-1}\ln N,    \quad x_i \leq d,\\
(b) \qquad \vert Z (x_i,t_j) \vert &\leq  C\prod _{n=N/2}^i \left(1+\frac{h_{n} }
{ \sqrt{2T\ve}}\right)^{-1} +CN^{-1}\ln N,
 \quad x_i \geq d.
\end{align*}
\end{theorem}
 \begin{proof}
(a) For $0\leq x_i \leq d$, consider the following barrier function
\begin{align*}
B (x_i,t_j) & := C\Phi (x_i)\Psi (t_j), \quad \hbox{where}
\\
\Phi (x_i) & :=  \frac{\prod _{k=1}^i \bigl(1+\frac{h_k}
{ \sqrt{2T\ve}}\bigr)  }{\prod _{k=1}^{N/2} \bigl(1+\frac{h_k}{ \sqrt{2T\ve}}\bigr)  }
\quad \hbox{and} \quad
\Psi (t_j)  := \left(1-\frac{\theta T \ln N}{ M}\right)^{-j}.
\end{align*}
The parameter $ \theta \geq 1$ is  specified below and $M$ and $N$ are sufficiently large so that
\[
0<c \leq 1-\frac{\theta T \ln N}{ M} \quad \hbox{and} \quad  \ln N \geq 1+\frac{1}{T}.
\]
Note that  $\Phi (0,t_j) \geq 0, \Phi (x_i,0) \geq0, \Phi(d,t_j) = C\Psi (t_j) \geq C >0$. In addition,
\begin{align*}
 \sqrt{2\ve } D_x^+ \Phi (x_i) &= \frac{1}{\sqrt{T}} \Phi (x_i),\quad D_t^-\Psi (t_j) = \theta  \ln N \Psi (t_j) > 0, \\
\sqrt{2\ve }\left(1+\frac{h_i}{ \sqrt{2T\ve}}\right)D_x^-\Phi (x_i)&= \frac{1}{\sqrt{T}}\Phi (x_i),
\\
 -\ve \delta _x^2 \Phi (x_i) & =-\frac{1}{T}\frac{ h_i}{h_i+h_{i+1}}
 \left(1+\frac{h_i}{ \sqrt{2T\ve}}\right)^{-1} \Phi (x_i)
 \\
 & \geq - \frac{1}{T}\Phi (x_i).
\end{align*}
 So, it follows that, when $\kappa (x_i,t_j) \geq 0$ and for $N$ sufficiently large
\begin{align*}
(-\ve \delta _x^2 +\kappa D^-_x+ g D^-_t ) B(x_i,t_j) & \geq \left(\theta  \ln N - \frac{1}{T}\right) B(x_i,t_j) \geq  B (x_i,t_j)
\\
&
\geq e^{-\frac{\vert d-x_i \vert}{2\sqrt{\ve T}}}.
\end{align*}
We need a modification to the argument if at any mesh point $\kappa (x_i,t_j) < 0$. From (\ref{def-A}),
\[
(-\ve \delta _x^2 +\kappa D^+_x+ g D^-_t )B (x_i,t_j) \geq  (-\ve \delta _x^2 -A (d-x_i)D^+_x+gD^-_t )B (x_i,t_j).
\]
  For the fine mesh points,  where $d-\sigma  _1\leq x_i < d$,
\begin{align*}
\bigl( -\ve \delta _x^2 +\kappa D_x^+ +  gD_t^-) B (x_i,t_j)
&\geq \left( -\frac{1}{T}  -   \frac{A  (d-x_i)}{\sqrt{2T\ve}} + \theta \ln N
\right)B (x_i,t_j) \\
&\geq \left( -\frac{1}{T} -    A \ln N + \theta \ln N
\right)B (x_i,t_j).
\end{align*}
Then, by choosing $ \theta \geq 1+A$, we get that
\[
\bigl( -\ve \delta _x^2 +\kappa D_x^+ +  gD_t^-)B (x_i,t_j) \geq B (x_i,t_j),\quad  x_i \in [d-\sigma  _1, d).
\]
On the coarse mesh where $ 0<x_i<d-\sigma  _1 $, then using the inequality
$nt\leq (1+t)^n, t \geq 0$,
\begin{align*}
\frac{(d-x_i)}{\sqrt{2T\ve}}\Phi (x_i) &= \frac{\sigma  _1}{\sqrt{2T\ve}}\Phi (x_i) + \Phi (d-\sigma _1+h) \frac{(x_{3N/8}-x_i)}{\sqrt{2T\ve}} \left(1+\frac{H}
{ \sqrt{2T\ve}} \right)^{-(3N/8-i)}\\ &\leq CN^{-1} \ln N.
\end{align*}
Then, for sufficiently large $N$ and $0< x_i < d- \sigma _1$,
\begin{align*}
( -\ve \delta _x^2 +\kappa D_x + gD_t^-) B (x_i,t_j)
&\geq \left( -\frac{1}{T}  -   \frac{A  (d-x_i)}{\sqrt{2T\ve}} + \theta \ln N
\right)B (x_i,t_j) \\
&\geq \left( -\frac{1}{T}   + \theta \ln N
\right)B (x_i,t_j)
 -   CN^{-1} \ln N.
\end{align*}
Finish using a discrete comparison principle with the barrier function $B (x_i,t_j) + C t_j N^{-1}\ln N $.

(b)
For $x_i \geq d$, consider the following barrier function
\begin{align*}
B_1(x_i,t_j) &:= C\Phi _1(x_i) \Psi _1(t_j), \quad \hbox{where} \\
\Phi _1 (x_i) &:= \prod_{n=N/2}^i \left(1+\frac{ h_{n} }
{ \sqrt{2T\ve}}\right)^{-1}\quad \hbox{and} \quad \Psi _1 (t_j)  :=
\left(1-\frac{\theta T \ln N}{\delta ^2 M}\right)^{-j},
\end{align*}
 and we further assume that
\[
0<c \leq 1-\frac{\theta T \ln N}{\delta ^2 M}.
\]
Note first that $
B_1(d,t_j)  \geq C >0, \, B_1(x_i,0), B_1(1,t_j) \geq 0.
$
In addition, we have that
\[
\sqrt{2\ve} D_x^-\Phi_1 (x_i) = - \frac{1}{\sqrt{T}}\Phi _1(x_i), \ \Phi_1(d) =1,\quad  -\ve \delta _x^2 \Phi _1(x_i)
 \geq - \frac{1}{T} \Phi (x_i).
\]
Note that if $ 1-\sigma <x_i<1$, then
\[
\frac{(x_i-d)}{\sqrt{2T\ve}}\Phi_1 (x_i) \leq  \frac{1-\sigma-d}{\sqrt{2T\ve}}\Phi_1 (1-\sigma) + \frac{x_i-(1-\sigma)}{\sqrt{2T\ve}}\Phi_1 (x_i) \leq CN^{-1}.
\]
Hence, for sufficiently large $N$ and all the mesh points where $d < x_i < 1$, we repeat the argument from part (a) to conclude that
\[
\bigl( -\ve \delta _x^2 +\kappa D_x +  g D_t^-) B_1 (x_i,t_j)  \geq B_1 (x_i,t_j).
\]
\end{proof}

\begin{theorem}\label{thm5}
 Assume~\eqref{T-limit}. For sufficiently large $N$ and $M \geq {\mathcal O}(\ln (N))$, the solution of (\ref{disc-Z}) satisfies the bounds
\begin{equation}\label{z-bnd}
\vert Z (x_i,t _j) - z  (x_i,t _j)\vert \leq C \bigl(N^{-1}(\ln N)^2 +CM^{-1} \bigr)
\end{equation}
\end{theorem}
\begin{proof}
 From Theorem~\ref{th_z_component} and using $e^{-\theta s^2} \leq e^{\frac{1}{4\theta }}e^{- s}$, we deduce that
\[
\vert z(x_i,t_j)  \vert \leq C E_\gamma (x_i,t_j)  \leq Ce^{-\frac{\vert d-x_i \vert}{2\sqrt{\ve T}}}.
\]
 Thus,
\[
\vert z(x_i,t_j)  \vert \leq C N^{-1}, \quad  x_i \notin (d-\sigma_1,d] \cup[d, d+\sigma _2).
\]
In addition, from Theorem \ref{thm4}  it  also follows that
 \begin{equation}\label{outside-layer-bnd}
\vert Z (x_i,t_j) \vert \leq CN^{-1}, \ x_i  \notin (d-\sigma_1,d] \cup[d, d+\sigma _2).
\end{equation}
 Then, using the triangular inequality estimate~\eqref{z-bnd} follows when $x_i  \notin (d-\sigma_1,d] \cup[d, d+\sigma _2).$  Hence we only now need to consider the error in the internal fine mesh.
Within the fine mesh $\vert \kappa(x_i,t_j) \vert \leq C \sigma _1$ and so
 for $ x_i \in (d-\sigma_1, d+\sigma _2)$,
\begin{align*}
\vert L^{N,M} \bigl( Z - z \bigr) (x_i,  t_j) \vert &\leq  CN^{-1}  \ln N +CM^{-1} +C\sqrt{\ve} (\sqrt{t_j} -\sqrt{t_{j-1}}) , \ x_i \neq d; \\
\vert L^{N,M} \bigl( Z - z \bigr) (d,  t_j) \vert &\leq C \frac{N^{-1}\ln N}{\sqrt{\ve}}.
\end{align*}
Consider the piecewise linear barrier function, $B(x_i)$ defined by
\[
B(d-\sigma_1)= B(d+\sigma _2)=0, \quad B(d)=1,
\]
and then we deduce the error bound using the discrete barrier fuction
 \[
C N^{-1}(\ln N)^2 (1+B(x_i))+CM^{-1} (t_j+\sqrt{\ve} \sqrt{t_j}) .
 \]
and the discrete maximum principle.
 \end{proof}

The main result of this paper can now be stated.
\begin{theorem} \label{th_a(x,t)} For sufficiently large $N$ and $M \geq {\mathcal O}(\ln (N))$,
If $Y$ is the solution of (\ref{discrete-problem}) and $y$ is the solution of (\ref{problem3}).
 Then,   the global approximation $\bar Y$ on $\bar Q$ generated by the values of $Y$ on $\bar Q^{N,M}$ and bilinear interpolation, satisfies
\begin{eqnarray*}
\Vert  \bar Y - y \Vert _{[0,1] \times [t_{j-1},t_j]} \leq C(N^{-1} (\ln N)^2 + M^{-1}  \ln M).
\end{eqnarray*}
\end{theorem}
\begin{proof}
By combining the bounds in  (\ref{smooth-bnd}), (\ref{sing-bnd}) and (\ref{z-bnd}), the error bound is established at the nodes  of the mesh $\bar Q^{N,M}$. In order to extend to the global error bound, combine the arguments in \cite[Theorem 3.12]{fhmos} with the interpolation bounds in \cite[Lemma 4.1]{styor4} and the bounds on the derivatives of the components $ v,  w, z $. Note that from \cite[Lemma 4.1]{styor4}, we only require the first time derivative of any component of  $y$ to be uniformly bounded.
 \end{proof}

\section{Modifications when source term  is present}\label{source-term}

Here we outline the modifications to the method and to the analysis when $b > 0$. The problem is (\ref{problem3a}), but the differential equation
(\ref{a(x,t)}) is replaced with
\begin{equation} \label{de+source}
-\ve \hat u_{ss} + \hat a \hat u_s +\hat b \hat u+\hat u_t = \hat f, \quad (s,t) \in (0,1) \times (0,T].
\end{equation}
In addition to all of the constraints imposed in  (\ref{problem3a}), we also assume that $\hat b \in   C^{2+\gamma} (\bar{\hat Q} ), \  \hat b\geq 0 $
and the additional constraint  $\hat b_s(d,0)=\hat b_{ss}(d,0)=0$.
As before,  $\Gamma ^*$ is  defined by
$
d'(t)=\hat a(d(t),t), \, d(0)=d.
$
The operator $\hat L_d$, given in   (\ref{Lddef}),  is redefined as
\[
\hat L_d \hat F := -\ve \hat F _{ss} +\hat a(d(t),t) \hat F_s +\hat b(d(t),t) \hat F +\hat F_t
\]
and we  introduce a new function
\[
I(t) := e^{-\int _{r=0}^t \hat b(d(r),r) \ dr }.
\]
Then $
\hat L_d (I \hat \psi _i )=0, i=0,1,2,3,4.
$
 We  redefine the function (\ref{def-y}) to be
\begin{align*}
\hat y (s,t) & := \hat u(s,t)   - 0.5 [\phi ](d) I(t) \hat \psi _0 (s,t); \quad \hbox{where} \\
\hat L\hat y & = \hat f+ 0.5 [\phi ](d) I(t) \bigl((\hat a(d(t),t)- \hat a(s,t) ) \frac{\partial \hat \psi _0}{\partial s} + (\hat b(d(t),t)- \hat b(s,t) ) \hat \psi _0 \bigr)
\end{align*}
The changes in the transformed problem (\ref{problem3})  are: Find $y$ such that
\begin{subequations} \label{de-problem3}
\begin{align}
{\cal{L}}y
& = g\left( f+ 0.5 [\phi ](d)  \frac{(a(d,t)- a(x,t))}{\sqrt{\ve \pi t}}{ I(t)}  e^{-\frac{g(x,t)(x-d)^2}{4\ve t}}\right) \nonumber \\
& +  0.5 [\phi ](d) \bigl(b(d,t)- b(x,t)\bigr) g(x,t) I(t) \psi _0(x,t) \\
{\cal{L}}y & :=  -\ve y_{xx} +\kappa (x,t)y_x+ g(x,t)(b(x,t)y+ y_t), \\
y(p,t) & =  - 0.5 [\phi ](d) I(t) \psi _0 (p,t),\ p=0,1, \  0 < t \leq T.
\end{align}
\end{subequations}
The discrete problem is defined as in (\ref{discrete-problem}).
In the proof of Theorem \ref{th_z_component}, the presence of the source term will only effect the discussion of the regularity of the component $z_p(x,t)$ in  Appendix C.
 In addition, the component $z_R$ is in the space  $C^{4+\gamma}(\bar Q^-)\cup C^{4+\gamma}(\bar Q^+)$, due to the  additional constraint imposed on $\hat b$, and then  the bounds (\ref{bds-zr}) are also satisfied.
Consequently, the  proof of Theorem \ref{th_a(x,t)} will still apply.

\section{ Numerical results} \label{sec:numerical}

In this section we present numerical results for two test examples.
The exact solution of both examples are unknown. We estimate the orders of global convergence $P_\ve^{N,M}$ and the orders of global parameter-uniform convergence $P^{N,M}$ using the two-mesh method \cite[Chapter 8]{fhmos}:
For each $\ve \in S:= \{2^0,2^{-1},\ldots,2^{-26}\}$, compute the solutions $Y^{N,M}$ and $Y^{2N,2M}$  with~\eqref{discrete-problem} on the Shishkin meshes $\bar Q^{N,M}$ and $\bar Q^{2N,2M}$. Then,
 calculate the maximum two-mesh global differences
$$
D^{N,M}_\ve:= \Vert \bar Y^{N,M}-\bar Y^{2N,2M}\Vert _{\bar Q^{N,M} \cup \bar Q^{2N,2M}}, \  \forall \ve \in S;
$$
where  $\bar Y^{N,M}$  denotes the bilinear interpolation of the discrete solution $Y^{N,M}$  on the mesh $\bar Q^{N,M}.$  For each $\ve \in S$ the orders of global  convergence $ P^{N,M}_\ve$ are estimated by
$$
 P^{N,M}_\ve:=  \log_2\left (\frac{D^{N,M}_\ve}{D^{2N,2M}_\ve} \right), \  \forall \ve \in S.
$$
The uniform  two-mesh global differences $D^{N,M}$ and the uniform orders of global convergence $ P^{N,M}$ are calculated by
$$
D^{N,M}:= \max_{\ve \in S} D^{N,M}_\ve, \quad  P^{N,M}:=  \log_2\left ( \frac{D^{N,M}}{D^{2N,2M}} \right).
$$

 \begin{example} \label{ex2}
  Consider the following  test problem
 \[
\begin{array}{l}
-\ve  \hat u_{ss} + \hat a(s,t) \hat u_s + \hat u_t= 4s(1-s)t+t^2, \quad (x,t) \in (0,1)\times (0,0.5], \\
 \hat u(s,0)=-2, 0 \leq x < 0.2, \quad   \hat u(s,0)=1, \ 0.2 \leq x \leq 1,         \\
 \hat u(0,t)=-2, \quad    \hat u(1,t) =1, \ 0 < t \leq 0.5,
\end{array}
\]
where
\[
\hat a(s,t)=(0.9^2-(s-0.2)^2)/4.
\]
\end{example}
Note that $\hat a_s(d,0)=0.$ The characteristic curve is
\[
 \hat d(t)=\frac{1.1-0.7e^{-9t/20}}{1+e^{-9t/20}}.
\]
In~\cite{arxiv} it is proved that the co-ordinate transformation~\eqref{map} is not needed in order to obtain a global approximation when $\hat a$ only depends on the variable $t.$ Hence, we first examine if this transformation is needed if $\hat a=\hat a(s,t)$. In Table~\ref{tb:ex2-NoTransform}, we see that, without the mapping, the method is not  parameter-uniform.

\begin{table}[h]
\caption{ Example~\ref{ex2}: Maximum two-mesh global differences and orders of convergence
using the scheme from~\cite{arxiv},  where the co-ordinate transformation ~\eqref{map} is not used}
\begin{center}{\tiny \label{tb:ex2-NoTransform}
\begin{tabular}{|c||c|c|c|c|c|c|c|}
 \hline  & N=M=32 & N=M=64 & N=M=128 & N=M=256 & N=M=512 & N=M=1024 & N=M=2048\\
 \hline $D^{N,M}$
&4.422E-02 &4.546E-02 &1.531E-02 &3.916E-02 &1.966E-02 &4.448E-02 &1.328E-02 \\
$P^{N,M}$ &-0.040&1.570&-1.355&0.994&-1.178&1.744&\\ \hline \hline
\end{tabular}}
\end{center}
\end{table}

Example~\ref{ex2} is now approximated with the numerical scheme~\eqref{discrete-problem} proposed in this paper. The computed approximations to $y$ and $\hat u$ are displayed in Figure~\ref{fig:ex2-solutions} and the maximum two-mesh global differences are given in Table~\ref{tb:ex2}. These numerical results are in agreement with Theorem~\ref{th_a(x,t)}.

 \begin{figure}[h!]
\centering
\resizebox{\linewidth}{!}{
	\begin{subfigure}[Approximation to $y$]{
		\includegraphics[scale=0.5, angle=0]{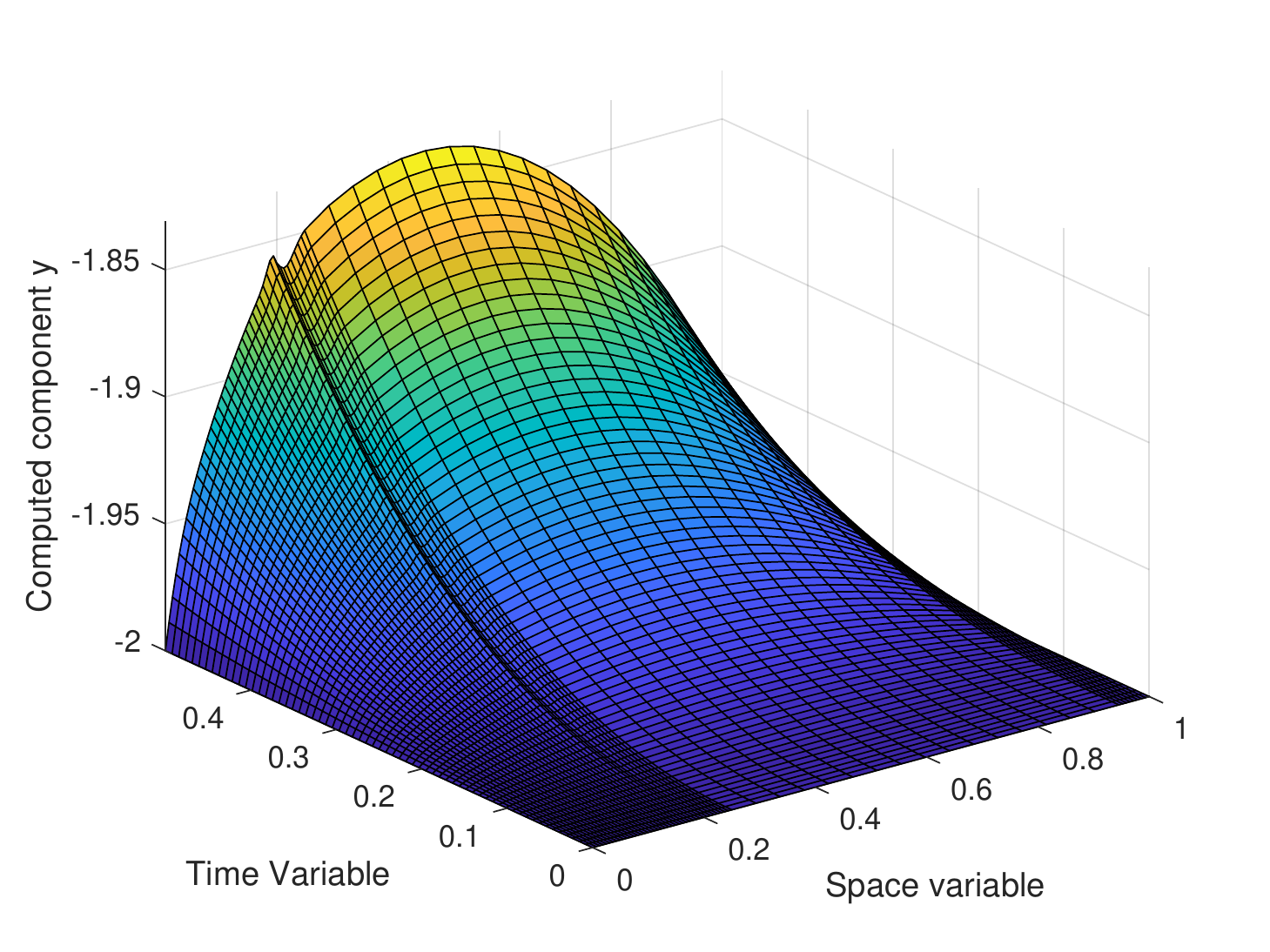}
		}
    \end{subfigure}
\begin{subfigure}[Approximation to $\hat u$]{
		\includegraphics[scale=0.5, angle=0]{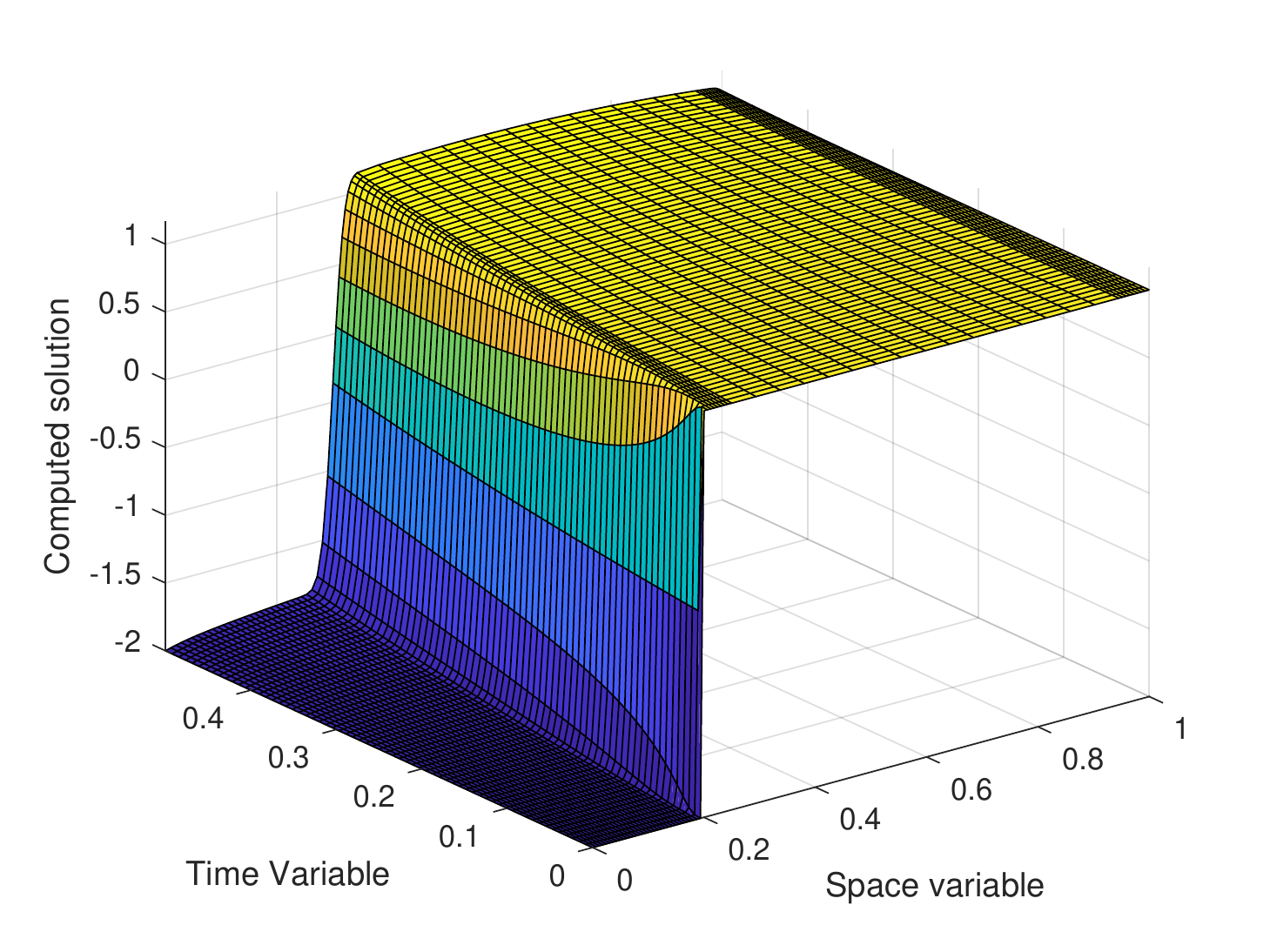}
		}
	\end{subfigure}
}
	\caption{Example~\ref{ex2}: Numerical approximations to $y$ and $\hat u$ with $\vr=2^{-12}$ and $N=M=64$}
	\label{fig:ex2-solutions}
 \end{figure}

\begin{table}[h]
\caption{Example~\ref{ex2}: Uniform two-mesh global differences and orders of convergence using the numerical method ~\eqref{discrete-problem}}
\begin{center}{\tiny \label{tb:ex2}
\begin{tabular}{|c||c|c|c|c|c|c|c|}
 \hline  & N=M=32 & N=M=64 & N=M=128 & N=M=256 & N=M=512 & N=M=1024 & N=M=2048
 \\
\hline \hline $\vr=2^{0}$
&3.503E-02 &{\bf 4.546E-02} &{\bf 1.531E-02} &{\bf 5.169E-03} &{\bf 2.067E-03} &{\bf 1.005E-03} &{\bf 4.955E-04} \\
&-0.376&1.570&1.567&1.322&1.041&1.020&
\\ \hline $\vr=2^{-1}$
&{\bf 4.422E-02} &1.495E-02 &5.041E-03 &2.017E-03 &9.795E-04 &4.827E-04 &2.396E-04 \\
&1.564&1.569&1.322&1.042&1.021&1.010&
\\ \hline $\vr=2^{-2}$
&1.426E-02 &4.795E-03 &1.927E-03 &9.318E-04 &4.585E-04 &2.274E-04 &1.132E-04 \\
&1.573&1.315&1.048&1.023&1.012&1.006&
\\ \hline $\vr=2^{-4}$
&1.986E-03 &7.580E-04 &3.886E-04 &1.967E-04 &9.897E-05 &4.964E-05 &2.486E-05 \\
&1.390&0.964&0.982&0.991&0.996&0.998&
\\ \hline $\vr=2^{-6}$
&8.317E-03 &3.022E-03 &9.091E-04 &3.251E-04 &1.625E-04 &8.126E-05 &4.063E-05 \\
&1.461&1.733&1.483&1.000&1.000&1.000&
\\ \hline $\vr=2^{-8}$
&1.610E-02 &8.733E-03 &3.419E-03 &1.081E-03 &3.076E-04 &1.008E-04 &4.369E-05 \\
&0.882&1.353&1.662&1.813&1.610&1.206&
\\ \hline $\vr=2^{-10}$
&1.325E-02 &9.919E-03 &5.841E-03 &2.769E-03 &1.111E-03 &4.467E-04 &1.897E-04 \\
&0.418&0.764&1.077&1.317&1.315&1.236&
\\ \hline $\vr=2^{-12}$
&9.178E-03 &5.996E-03 &3.206E-03 &1.437E-03 &6.355E-04 &3.306E-04 &1.718E-04 \\
&0.614&0.903&1.158&1.177&0.943&0.945&
\\ \hline $\vr=2^{-14}$
&6.754E-03 &4.265E-03 &2.232E-03 &1.121E-03 &6.165E-04 &3.434E-04 &1.895E-04 \\
&0.663&0.934&0.994&0.863&0.844&0.858&
\\
\hline &\vdots &\vdots &\vdots &\vdots &\vdots &\vdots &\vdots \\
&&&&&&&
\\ \hline $\vr=2^{-24}$
&5.397E-03 &3.502E-03 &1.916E-03 &1.130E-03 &6.769E-04 &3.823E-04 &2.149E-04 \\
&0.624&0.870&0.762&0.739&0.824&0.831&
\\ \hline $\vr=2^{-26}$
&5.396E-03 &3.501E-03 &1.916E-03 &1.130E-03 &6.770E-04 &3.823E-04 &2.149E-04 \\
&0.624&0.870&0.761&0.739&0.824&0.831&
\\
 \hline $D^{N,M}$
&4.422E-02 &4.546E-02 &1.531E-02 &5.169E-03 &2.067E-03 &1.005E-03 &4.955E-04 \\
$P^{N,M}$ &-0.040&1.570&1.567&1.322&1.041&1.020&\\ \hline \hline
\end{tabular}}
\end{center}
\end{table}

\begin{example} \label{ex3}
Consider  the test problem
 \[
\begin{array}{l}
-\ve  \hat u_{ss} + (1+s^2) \hat u_s + (s+t) \hat u+ \hat u_t= 4s(1-s)t+t^2, \quad (x,t) \in (0,1)\times (0,0.5], \\
 u(x,0)=-2, 0 \leq x < 0.1, \quad    u(x,0)=1, \ 0.1 \leq x \leq 1,         \\
 u(0,t)=-2, \quad    u(1,t) =1, \ 0 < t \leq 0.5.
\end{array}
\]
\end{example}
Note that the source term is present in this example and then problem~\eqref{de-problem3} is approximated with the numerical method~\eqref{discrete-problem} on the Shishkin mesh $\bar Q^{N,M}.$ For this example, we have
\[
I(t)
=(\cos t-0.1\sin t) e^{-t^2/2}.
\]
In addition, observe that $\hat a_s(d,0) \ne 0$ and $\hat b_s(d,0)\ne 0$.
In Table~\ref{tb:ex3} we see that the numerical approximations   converge with almost first order.

\begin{table}[h]
\caption{ Example~\ref{ex3}: Maximum two-mesh global differences and orders of convergence  using the numerical method ~\eqref{discrete-problem}}
\begin{center}{\tiny \label{tb:ex3}
\begin{tabular}{|c||c|c|c|c|c|c|c|}
 \hline  & N=M=32 & N=M=64 & N=M=128 & N=M=256 & N=M=512 & N=M=1024 & N=M=2048 \\
\hline \hline $\vr=2^{0}$
&1.978E-01 &6.835E-02 &3.224E-02 &4.361E-02 &1.478E-02 &4.979E-03 &1.984E-03 \\
&1.533&1.084&-0.436&1.561&1.570&1.328&
\\ \hline $\vr=2^{-2}$
&2.516E-02 &3.521E-02 &1.235E-02 &4.115E-03 &1.624E-03 &7.870E-04 &3.880E-04 \\
&-0.485&1.511&1.585&1.341&1.045&1.020&
\\ \hline $\vr=2^{-4}$
&1.631E-01 &7.441E-02 &3.434E-02 &1.690E-02 &8.342E-03 &4.147E-03 &2.069E-03 \\
&1.132&1.116&1.023&1.018&1.008&1.003&
\\ \hline $\vr=2^{-6}$
&3.309E-01 &2.423E-01 &{\bf 1.616E-01} &7.870E-02 &3.916E-02 &1.960E-02 &9.822E-03 \\
&0.449&0.585&1.038&1.007&0.998&0.997&
\\ \hline $\vr=2^{-8}$
&3.127E-01 &2.284E-01 &1.421E-01 &7.545E-02 &4.301E-02 &{\bf 2.506E-02} &1.398E-02 \\
&0.453&0.684&0.913&0.811&0.780&0.842&
\\ \hline $\vr=2^{-10}$
&3.347E-01 &2.251E-01 &1.399E-01 &7.839E-02 &4.277E-02 &2.472E-02 &1.390E-02 \\
&0.572&0.686&0.836&0.874&0.791&0.830&
\\ \hline $\vr=2^{-12}$
&{\bf 3.583E-01} &2.367E-01 &1.448E-01 &8.242E-02 &4.451E-02 &2.484E-02 &1.403E-02 \\
&0.598&0.709&0.813&0.889&0.842&0.824&
\\ \hline $\vr=2^{-14}$
&3.460E-01 &{\bf 2.450E-01} &1.477E-01 &8.420E-02 &4.566E-02 &2.492E-02 &1.408E-02 \\
&0.498&0.730&0.811&0.883&0.873&0.824&
\\ \hline $\vr=2^{-16}$
&3.071E-01 &2.240E-01 &1.523E-01 &{\bf 8.540E-02} & {\bf 4.605E-02} &2.495E-02 &1.410E-02 \\
&0.455&0.557&0.834&0.891&0.884&0.824&
\\
\hline &\vdots &\vdots &\vdots &\vdots &\vdots &\vdots &\vdots \\
&&&&&&&
\\ \hline $\vr=2^{-24}$
&3.111E-01 &2.249E-01 &1.391E-01 &7.423E-02 &4.296E-02 &2.495E-02 &1.410E-02 \\
&0.468&0.693&0.906&0.789&0.784&0.823&
\\ \hline $\vr=2^{-26}$
&3.117E-01 &2.248E-01 &1.391E-01 &7.423E-02 &4.296E-02 &2.495E-02 &{\bf 1.410E-02} \\
&0.471&0.692&0.906&0.789&0.784&0.823&
\\ \hline $D^{N,M}$
&3.583E-01 &2.450E-01 &1.616E-01 &8.540E-02 &4.605E-02 &2.506E-02 &1.410E-02 \\
$P^{N,M}$ &0.548&0.601&0.920&0.891&0.878&0.829&\\ \hline \hline
\end{tabular}}
\end{center}
\end{table}

\section{Appendix A: A set of singular functions}

 In this appendix the singular functions~$\hat \psi_i, \, i=0,1,2,3,4$
 are defined and bounds of their derivatives are given. These functions are the main terms in the regularity  expansion~\eqref{expansion2} of the continuous solution $\hat u(s,t)$. These bounds are used in the truncation error analysis of the interior layer component $z$.

For any function $F(x,t)=\hat F(s,t)$ we have
\begin{align*}
{\cal{L}} F &= -\ve  F_{xx} + \kappa F _x +g F_t =
 -\ve g\hat F_{ss} + \left(\sqrt{g} \kappa + g \frac{\partial s}{\partial t} \right)\hat F _s +g\hat F_t \\
&= g \hat L_d \hat F +g \left(\frac{\kappa}{\sqrt{g}}  +  \frac{\partial s}{\partial t} -a(d,t) \right)\hat F _s,
\end{align*}
where
\begin{equation}\label{Lddef}
\hat L_d \hat F  := -\ve \hat F_{ss} + \hat a(d(t),t) \hat F _s +\hat F_t.
\end{equation}
Hence, from (\ref{map}), (\ref{coeff-problem3})  and using $\hat a(d(t),t)=a(d,t)$ we have
\begin{equation}\label{handy}
{\cal{L}} F(x,t) = g \hat L_d \hat F + \sqrt{g} (a(x,t)-a(d,t)) \frac{\partial F}{\partial x}.
\end{equation}

 We will define a set of functions $\{ \hat \psi _i \} _{i=0}^4$ such that $\hat L_d \hat \psi _i=0$; $ \hat \psi _i \in C^{i-1} (\bar {\hat  Q}), \ i \geq 1$. Each function $\hat \psi _i$ is smooth within the open region $ \hat Q \setminus \Gamma ^*$. } Define the two singular functions~\cite{bobisud}
\begin{equation}\label{zero}
\hat \psi _0(s,t) := \erfc \left ( \frac{d(t)-s}{2\sqrt{\ve t}} \right) ,\quad \hat E(s,t) := e^{-\frac{(s-d(t))^2}{4\ve t}}.
\end{equation}
Then we explicitly write out the derivatives of these two functions
\begin{align*}
\frac{\partial \hat \psi _0}{\partial s} &=  \frac{1}{\sqrt{\ve \pi t}} \hat E  , \quad  \frac{\partial \hat E}{\partial s} =  \frac{d(t)-s}{2\ve t} \hat E,\quad \frac{\partial \hat E}{\partial t}  = \frac{\sqrt{ \pi }(d(t)-s)}{2\sqrt{\ve  t}}\frac{\partial \hat \psi _0}{\partial t};\\
\ve \frac{\partial ^2\hat \psi _0}{\partial s ^2} &=  \frac{d(t)-s}{2t\sqrt{\ve \pi t}} \hat E, \quad \frac{\partial \hat \psi _0}{\partial t} =  \frac{1}{\sqrt{\ve \pi t}}  \left(\frac{(d(t)-s)}{2t}- \hat a(d(t),t) \right)\hat E.
\end{align*}
Hence,  we have that
$
\hat L_d \hat \psi _0=0.
$
Observe that
\begin{align*}
\hat L_d ((d(t)-s)\hat \psi _0) &=2\ve \frac{\partial \hat \psi _0}{\partial s}= 2\frac{\sqrt{\ve}}{\sqrt{\pi t}}\hat E; \quad
\hat L_d \hat E=\frac{1}{2t}  \hat E
\\ \hbox{and} \quad
\hat L_d (t^{n+0.5}\hat E) &=(n+1) t^{n-0.5} \hat E \quad \hbox{for all } n \geq 0.
\end{align*}
We now  define the remaining weakly singular functions:
\begin{subequations}
\begin{eqnarray}
\hat \psi _1(s,t) := (d(t)-s)\hat \psi _0 -2\frac{\sqrt{\ve t}}{\sqrt{\pi}} \hat E, \label{one} \\
\hat \psi _i = (d(t)-s) \hat \psi _{i-1}+2\ve t (i-1) \hat \psi _{i-2},\qquad  i=2,3,4;\label{singular-functions}
\end{eqnarray}
\end{subequations}
 which satisfy
\begin{align*}
\frac{\partial \hat \psi _i}{\partial s} &= -i \hat \psi _{i-1},\quad \hat L_d \hat \psi _i= 0, \qquad  i=1,2,3,4 ;\\
\frac{(-1)^i}{i!} \Bigl[  \frac{\partial ^i \hat \psi _i}{\partial s ^i } \Bigr](d,0) &=2, \qquad i=0,1,2,3,4;\\
(d(t)-s)\frac{\partial \hat \psi _i}{\partial s} &= -i \hat \psi _{i}+2\ve t i(i-1) \hat \psi _{i-2} \in C^{i-1} (\bar {\hat  Q}),\qquad  i=2,3,4.
\end{align*}
Define the  parameterized exponential function
\[
\hat E_\gamma (s,t) := e^{-\frac{\gamma (s-d(t))^2}{4\ve t}},\qquad  0 < \gamma < 1.
\]
 Using the inequality $\erfc(z) \leq C e^{-z^2} \leq C e^{\gamma ^2/4}e^{-\gamma z}, \,  \forall z \ge 0,$
it follows that
\begin{equation}\label{bound-a(t)}
 \Bigl \vert \frac{\partial ^j }{\partial t ^j } \hat \psi _0(s,t)\Bigr \vert, \Bigl \vert \frac{\partial ^j }{\partial t ^j } \hat E(s,t)\Bigr \vert   \leq  C
 \left( \frac{1}{t} +\frac{1}{\sqrt {\ve t}}\right)^j\hat E_\gamma (s,t);\quad j=1,2;
\end{equation}

 Based on the map (\ref{map}) and the definition of the function $g$ (\ref{def-g}) we have
\[
d(t)-s =  \sqrt{g} (d-x) \quad \hbox{and} \quad  \frac{\partial \psi _i }{\partial x} (x,t) =
\sqrt{g} \frac{\partial \hat \psi _i }{\partial s } (s,t).
\]
In the transformed domain, the two fundamental functions are:
\begin{eqnarray*}
 \psi _0(x,t) := \erfc \left( \frac{\sqrt{g(t)}(d-x)}{2\sqrt{\ve t}} \right) ,\qquad  E(x,t) := e^{-\frac{g(t)(x-d)^2}{4\ve t}}.
\end{eqnarray*}
It follows that
\begin{align*}
& \Bigl \vert \frac{\partial ^j }{\partial t ^j }  \psi _0(x,t)\Bigr \vert ,
 \Bigl \vert \frac{\partial ^j }{\partial t ^j }  E(x,t)\Bigr \vert    \leq C
 \left( 1 +\frac{ 1}{ t}\right)^j E_\gamma (x,t), \quad j=1,2, \quad x\neq d; 
 \\
& \vert \psi _0(x,t) \vert \leq C, \
\Bigl \vert \frac{\partial ^i }{\partial x ^i }  \psi _0(x,t)\Bigr \vert  , \Bigl \vert \frac{\partial ^i }{\partial x ^i }E(x,t)\Bigr \vert    \leq
C\left(\frac{1}{\sqrt {\ve t}}\right) ^{i} E_\gamma (s,t), \quad 1\leq i \leq 4.
\end{align*}
Observe that the bounds on the time derivatives of these two functions do not depend adversely on the singular perturbation parameter $\ve$. This contrasts with the bounds on the time derivatives of these functions in the original variables $(s,t)$.

In the transformed variables, we  see from (\ref{handy}) that
\[
{\cal{L}}\psi _i = \sqrt{g} (a(x,t)-a(d,t)) \frac{\partial \psi _i}{\partial x}\neq 0,\quad \hbox{for} \quad x \neq d.
\]
 The fact that ${\cal{L}}\psi _i\ne 0$, when $\hat a$ depends on the spatial variable, results in the function $\hat y$ exhibiting an interior layer (see Remark~\ref{rem:z}.)

The next  singular function is
\begin{subequations}\begin{eqnarray}\label{one-def}
 \psi _1(x,t) :=  \sqrt{g}(d-x) \psi _0 -2\frac{\sqrt{\ve t}}{\sqrt{\pi}} E \quad \hbox{and} \quad \frac{\partial \psi_1}{\partial x }   =-\sqrt{g} \psi _0
\end{eqnarray}
and the subsequent three functions\footnote{The functions $ \psi _1, \psi _2 $ were defined earlier by Shishkin in \cite[(4.8c)]{shishkin4}  and Bobisud in~\cite{bobisud}}  are
\begin{equation}\label{two+three-def}
\psi _i(x,t)  := \sqrt{g}(d-x) \psi _{i-1} +2\ve t (i-1)\psi _{i-2}, \quad i=2,3,4.
\end{equation}
\end{subequations}
As the first space derivatives of these functions are involved in the analysis of the interior layer function, we explicitly record that
\[
\frac{\partial \psi _n}{\partial x} =-i\sqrt{g} \psi _{n-1}, \, n=2,3,4  \ \hbox{ and } \ (d-x)^it^j \psi _0 \in C^{2+\gamma}(\bar Q),  \hbox{ if }  i+2j \geq 3.
\]
For these singular functions\footnote{For $n=0,1,2,3,4,\  \psi _n(x,0) = 2(d-x)^n,\ x > d; \ \psi _n(x,0) =0,\  x < d. $}, we can  establish the bounds
\begin{align*}
   \Bigl \vert \frac{\partial ^j }{\partial t ^j }  \psi _1(x,t)\Bigr \vert & \leq C
 \left( 1 +\frac{ \sqrt{\ve}}{ \sqrt{t}}\right)^j E_\gamma +C , \quad j=1,2;\\
   \Bigl \vert \frac{\partial  }{\partial x  }  \psi _1(x,t)\Bigr \vert     & \leq C, \quad
\Bigl \vert \frac{\partial ^i }{\partial x ^i }  \psi _1(x,t)\Bigr \vert      \leq
C\left(\frac{1}{\sqrt {\ve t}}\right) ^{i-1} E_\gamma +C, \quad  2 \leq i \leq 4,
\end{align*}
and
\[
\Bigl \vert \frac{\partial  }{\partial t  }  \psi _n(x,t)\Bigr \vert \leq C,\quad \Bigl \vert \frac{\partial ^2 }{\partial x ^2 }  \psi _n(x,t)\Bigr \vert \leq  C; \quad n=2,3,4;
\]
 on the second time derivatives
\begin{align*}
\Bigl \vert \frac{\partial ^2 }{\partial t ^2 }  \psi _2(x,t)\Bigr \vert  & \leq
C \left (1+ \sqrt{\frac{\ve}{t}}\right)^2 E_\gamma (x,t)+C, \\
\Bigl \vert \frac{\partial ^2 }{\partial t ^2 }  \psi _3(x,t)\Bigr \vert   & \leq
C   \left (1+ \ve \sqrt{\frac{\ve}{t}}\right) E_\gamma (x,t)+C;\quad
\Bigl \vert \frac{\partial ^2 }{\partial t ^2 }  \psi _4(x,t)\Bigr \vert   \leq
C   E_\gamma (x,t)+C;
\end{align*}
 on the fourth space derivatives
\begin{eqnarray*}
 \Bigl \vert \frac{\partial ^4 }{\partial x ^4 }  \psi _j(x,t)  \Bigr \vert \leq C(\sqrt{\ve t})^{j-4} E_\gamma (x,t)+C, \quad j=2,3,4;
\end{eqnarray*}
 and on the third space derivatives
\begin{eqnarray*}
 \Bigl \vert \frac{\partial ^3 }{\partial x ^3 }  \psi _2(x,t)  \Bigr \vert \leq  C
 \left( 1 +\frac{1}{\sqrt {\ve t}}\right)  E_\gamma (x,t)+C; \quad
\Bigl \vert \frac{\partial ^3 }{\partial x ^3 } \psi _n(x,t)\Bigr \vert    \leq  C, \quad n=3,4.
\end{eqnarray*}
 One can check that for all $m \geq 0, n \geq 1$
\begin{subequations}
\begin{align}
&\hat L_d(t^{m+0.5}\hat E)=(m+1) t^{m-0.5} \hat E \label{express0}\\
&\hat L_d (t^m(d(t)-s)^n\hat \psi _0)  \nonumber \\
&\hspace{2cm} =\bigl(mt^{m-1} (d(t)-s)^n  -\ve n(n-1)t^{m} (d(t)-s)^{n-2}\bigr)\hat \psi _0 \nonumber \\
&\hspace{2cm} +2\sqrt{\frac{\ve }{\pi t}} nt^{m} (d(t)-s)^{n-1}  \hat E \label{express1}
\\
&\hat L _d(\sqrt{t}(d(t)-s)^n \hat E) \nonumber \\
& \hspace{2cm} =\left( \frac{(n+1)(d(t)-s)^{n}}{\sqrt{ t}} - \ve n(n-1)\sqrt{ t}(d(t)-s)^{n-2}\right) \hat E \label{express2}\\
&\hat L_d (t\sqrt{t}(d(t)-s) \hat E)=3(d(t)-s)\sqrt{ t} \hat E \label{express3}.
\end{align}
\end{subequations}
These expressions will be used to deduce bounds for the component $z_p$   in the decomposition~\eqref{decomposition_z} of $z$.

 In addition, we assume that $a_x(d,0)=0$. This  guarantees that the component $z_c$ of $z$ in~\eqref{decomposition_z} satisfies $z_c\in C^{4+\gamma}(\bar Q^-)\cup C^{4+\gamma}(\bar Q^+).$ The regularity of this component comes from observing that $\vert g-1 \vert \leq Ct$ and so \[
(d-x)(g-1), \ t(g-1) \psi _0,  \ (d-x)^2(g-1) \psi _0 \in C^{2+\gamma}(\bar Q).
\]
 Thus, from~(\ref{handy}), \eqref{one-def}, \eqref{two+three-def} and the assumption $a_x(d,0)=0$, we have
\[
{\cal{L}} \psi _i = -ig\bigl(a(d,t)- a(x,t)\bigr) \psi _{i-1}\in C^{2+\gamma}(\bar Q), \ i\geq 2,
\]
which is used in~\eqref{RegZc} in  Appendix C.

\section{Appendix B:  Decomposition of the solution}

In this appendix we decompose the solution of problem~\eqref{problem3a} into a regular $\hat v$, boundary layer $\hat w$ and interior layer $\hat z$ components. Bounds for the derivatives of $\hat v$ and $\hat w$ are established here and the bounds for the component $z$ in Appendix C.

We have the following expansion for the solution of problem (\ref{problem3a}):
\begin{equation}\label{expansion2}
\hat u (s,t) = 0.5 \sum _{i=0}^4 [\phi ^{(i)} ] (d)  \frac{(-1)^i}{i!}  \hat \psi _i (s,t) + \hat R(s,t), \qquad \hat R \in C^{4+\alpha}(\bar {\hat Q});
\end{equation}
and, as we have assumed that  $[\phi ' ] (d)=0 $, then
\[
\hat y (s,t) = 0.5 \sum _{i=2}^4 [\phi ^{(i)} ] (d)  \frac{(-1)^i}{i!}  \hat \psi _i (s,t) + \hat R(s,t).
\]
Note that the smooth  remainder $\hat R$ satisfies the singularly perturbed problem
\begin{align*}
\hat L \hat R &=\hat f - 0.5 \sum _{i=0}^4 [\phi ^{(i)} ] (d)  \frac{(-1)^i}{i!}  \hat L \hat \psi _i (s,t), \ (s,t) \in \hat Q;\\
\hat R(s,0) &= \hat y(s,0) - 0.5 \sum _{i=2}^4 [\phi ^{(i)} ] (d)  \frac{(-1)^i}{i!}  \hat \psi _i (s,0), \ 0 \leq s \leq 1; \\
\hat R(p,t) &= \hat y(p,t) - 0.5 \sum _{i=2}^4 [\phi ^{(i)} ] (d)  \frac{(-1)^i}{i!}  \hat \psi _i (p,t),\quad t >0; \ p=0,1.
\end{align*}
This can be further decomposed as follows
\begin{eqnarray*}
\hat R = \hat v + \hat w + \hat z, \quad \hat v, \hat w \in C^{4+\alpha}( \bar {\hat Q});
\end{eqnarray*}
where
\begin{align*}
\hat L \hat v &= \hat f,  \quad \hat L \hat w = 0, \quad \hat L \hat z = \hat L \hat R  - \hat f; \quad (s,t) \in  \hat Q;\\
\hat v (0,t) &= \hat R (0,t), \quad \hat v(s,0) = \hat R(s,0), \quad \hat v(1,t) = \hat v^*(1,t); \\
\hat w(0,t) &= 0, \quad \hat w(s,0)= 0, \quad \hat w(1,t) = (\hat R -\hat v^*)(1,t); \\
\hat z(0,t) &= 0,\quad  \hat z(s,0) = 0,\quad  \hat z(1,t) = 0.
\end{align*}
As in \cite{dervs-parabolic-conv-diff}, the outflow boundary values for $\hat v$ can be specified (they are denoted by $\hat v^*(1,t)$ above) so that we have the following bounds
\begin{align*}
\left \vert \frac{\partial ^{i+j}   }{\partial s ^i \partial t ^j } \hat v (s,t)\right \vert & \leq  C,\  0 \leq i+j \leq  2, \qquad
\left \vert \frac{\partial ^{3}   }{\partial s ^3  } \hat v(s,t)\right \vert  \leq  C \left(1+\frac1{\ve}\right),  \\
\left \vert \frac{\partial ^{i+j} }{\partial s ^i \partial t ^j } \hat w(s,t)\right \vert  & \leq  C\ve ^{-i}(1+\ve ^{1-j}) e^{-\alpha(1-s)/\ve}, \quad 0 \leq i+2j \leq  4.
\end{align*}

\section{Appendix C: Regularity and bounds on the interior layer function}
To obtain sharp bounds on the derivatives of  the interior layer component $\hat z$, we transform the problem $ \hat L \hat z (s,t)= \hat L \hat R (s,t)-\hat f $ to
the $(x,t)$ coordinate system.
 In this appendix bounds for the two subcomponents $z_c$ and $z_p$ in the decomposition~\eqref{decomposition_z} of $z$ are established. In the case of the component $z_p$, they are established using a further decomposition into two components $z_q$ and $z_R$.
By the definition (\ref{zc-def}) of the subcomponent $z_c$, we have
\begin{align}
{\cal{L}} z_c(x,t)&=- \frac{g}{2}\bigl(a(d,t)- a(x,t)\bigr) \sum _{i=1}^3 [\phi ^{(i+1)} ] (d)  \frac{(-1)^i}{i!}   \psi _i (x,t) \nonumber\\
 &=: F_c(x,t) \in C^{2+\gamma}(\bar Q); \label{RegZc}
\end{align}
and, hence,
\[
z_c \in C^{4+\gamma}(\bar Q^-)\cup C^{4+\gamma}(\bar Q^+).
\]
The  function $z_c$ is sufficiently regular within each sub-domain to allow us use results from \cite{ladyz} to bound the derivatives of $z_c$.
In the stretched variable
\[
\zeta = \frac{x-d}{\sqrt{\ve}},
\]
we have  the bounds
\begin{subequations}\label{bds-zc}
\begin{align}
\vert {\cal{L}}z_c  (\zeta ,t) \vert
&\leq C \sqrt{\ve} e^{-\frac{\gamma g \zeta ^2}{4 t}},
\\
 \Bigl \vert \frac{\partial ^{i+j} ({\cal{L}}z_c  (\zeta ,t) )}{\partial \zeta ^i \partial t ^j } \Bigr \vert &\leq C \vr^{i/2} e^{-\frac{\gamma g \zeta ^2}{4 t}}, \quad 1 \leq i+2j \leq 2.
\end{align}
\end{subequations}
These bounds are used in Theorem~\ref{th_z_component} to deduce estimates for the component $z_c$ and some of its partial derivatives.

We next examine the regularity of the subcomponent $z_p(x,t)$, which is defined as the solution of problem  (\ref{z-def}). From  assumption (\ref{simplify}) we have the following  Taylor expansion
\begin{align*}
&a(d,t)-a(x,t) = p_d(x,t)+r_1(x,t), \\
&p_d(x,t):= - \left[ a_{xx}(d,0)\frac{(d-x)^2}{2!}+a_{xxx}(d,0)\frac{(d-x)^3}{3!}+t(d-x)a_{xt}(d,0) \right] , \\
&r_1(x,t) := K_0(d-x)^4+ K_1t(d-x)^2 +K_2t (d-x)^3+K_3 t^2(d-x).
\end{align*}
Once again, this interior layer component $z_p$  is decomposed into the sum
\begin{align}
z_p(x,t)&:=z_q(x,t)+  z_R(x,t), \label{DecompZP}\\
z_q(x,t)&=  B_1\sqrt{t}(g(d-x)^2+\ve t)E,  \nonumber \\
&+B_2\sqrt{gt}(d-x)(g(d-x)^2+2\ve t )E+B_3t\sqrt{gt}(d-x)E. \label{ZQ}
\end{align}
The constants  $B_1,B_2$ and $B_3$ are given by
\[
B_1 :=- \frac{a_{xx}(d,0)}{3!\sqrt{\ve \pi}}, \ B_2 := \frac{a_{xxx}(d,0)}{4!\sqrt{\ve \pi}},\ B_3 := \frac{a_{xt}(d,0)}{3\sqrt{\ve \pi}}.
\]
Note that $z_q \in C^{2+\gamma}(\bar Q^-)\cup C^{2+\gamma}(\bar Q^+)$ and
\[
z_q(d,t)=\ve B_1t\sqrt{t}, \ [z_q](d,t) =0, \quad \Bigl[\frac{1}{\sqrt{g}} \frac{\partial z_q}{\partial x} \Bigr](d,t) =0, \quad z_q(x,0)=0.
\]
Using that $\erfc(z) \leq C e^{-z^2}, \forall z$, we can  establish the bounds
\begin{subequations}\label{ZQxt}
\begin{align}
\Bigl \vert \frac{\partial ^{i} z _q}{\partial x^i   } (x,t)\Bigr \vert &\leq C(\sqrt{\ve})^{-i}\bigl( \sqrt{\ve }+\vert a_{xt}(d,0)\vert \bigr)E_\gamma (x,t), \quad i=0,1,2,3; \label{ZQx}\\
\Bigl \vert \frac{\partial ^j z _q}{ \partial t ^j} (x,t)\Bigr \vert &\leq C(\vert a_{xt}(d,0)\vert+\sqrt{\ve}(\sqrt{t})^{1-j})E_\gamma (x,t), \quad j=1,2, \label{ZQt} \\
\Bigl \vert \frac{\partial^{2} z _q}{\partial x \partial t   } (x,t)\Bigr \vert & \leq  C(\sqrt{\ve})^{-1}\bigl( \sqrt{\ve }+\vert a_{xt}(d,0)\vert \bigr)E_\gamma (x,t).
\end{align}
\end{subequations}

By the choice of constants  $B_1, B_2, B_3$ and using the expressions (\ref{express0}), (\ref{express2}), (\ref{express3})  (\ref{handy}), we see that  $z_p(x,t)$ satisfies
\begin{align}
{\cal{L}} z_p(x,t)   &= p_d(x,t) \frac{g}{\sqrt{\ve \pi t}} E(x,t)  +r_2(x,t) + {\cal{L}} z_R(x,t), \quad \hbox{where} \nonumber \\
r_2(x,t)&:= ( \tilde p_d- p_d)(x,t)\frac{g}{\sqrt{\ve \pi t}} E(x,t)  +\sqrt{g}(a(x,t)-a(d,t))\frac{\partial z_q}{\partial x} \label{r2}
\end{align}
 and
\[
\tilde p_d(x,t):=- \left[a_{xx}(d,0)\frac{g(d-x)^2}{2!}+a_{xxx}(d,0)\frac{g\sqrt{g}(d-x)^3}{3!}+t\sqrt{g}(d-x)a_{xt}(d,0) \right].
\]
The function $p_d$ and the related function $\tilde p_d$ satisfy
\[
\vert \tilde p_d(x,t) -  p_d(x,t) \vert \leq C\vert (g-1)(d-x) \vert \bigl(\vert d-x \vert +t \vert a_{xt}(d,0) \vert \bigr).
\]
By the definitions (\ref{z-def}) of $z$  and (\ref{zc-def}) of the subcomponent $z_c$ we have
\begin{align*}
{\cal{L}}z_R 
&=  r_1(x,t) \frac{g}{\sqrt{\ve \pi t}}E(x,t)  -r_2(x,t)
=: F_R(x,t) \in C^{2+\gamma}(\bar Q); \\
z_R(x,0) &= 0 ; \
[z_R](d,t)=0 , \quad [(z_R)_x](d,t)=0;
\\
z_R(0,t) &= K_4\frac{\sqrt{t}}{\sqrt{\ve }}e^{-\frac{ g d^2}{\ve t}},  z_R(1,t)  =K_5\frac{\sqrt{t}}{\sqrt{\ve }}e^{-\frac{ g (1-d)^2} {\ve t}},
\end{align*}
where the values of $K_4$ and $K_5$ can be obtained from~\eqref{ZQ}. Hence, the  function $z_R$ is sufficiently regular within each sub-domain to allow us use results from \cite{ladyz} to bound the derivatives of $z_R$. That is,
\[
z_R \in C^{4+\gamma}(\bar Q^-)\cup C^{4+\gamma}(\bar Q^+).
\]
Moreover, using  $\vert \sqrt{g}-1 \vert \le \vert g-1 \vert$ and
$
 \vert g-1 \vert \leq Ct,
$
 we conclude that,
\[
 \vert {\cal{L}}z_R  (x,t) \vert \leq C  E_\gamma (x,t) \leq C  e^{-\frac{\vert x-d\vert}{2\sqrt {\ve t}}}, \quad x\neq d.
\]
In the stretched variable
\[
\zeta = \frac{x-d}{\sqrt{\ve}},
\]
we have that
\begin{equation}\label{bds-zr}
 \Bigl \vert \frac{\partial ^{i+j} ({\cal{L}}z_R  (\zeta ,t) )}{\partial \zeta ^i \partial t ^j} \Bigr \vert \leq C  e^{-\frac{\gamma g \zeta ^2}{4 t}},  \quad 0 \leq i+2j \leq 2.
\end{equation}

\end{document}